\documentclass[a4paper,11pt,english]{smfart}
\usepackage{amsfonts,skt}
\usepackage{amsmath} 
\usepackage{latexsym}
\usepackage{array}
\usepackage{amssymb}
\usepackage{smfthm}
\usepackage{mathrsfs}
\usepackage[T1]{fontenc}
\usepackage{pdfsync}
\theoremstyle{plain}

\newtheorem*{acknowledgements}{Acknowledgements}

\newcommand{\R}{  \mathbb{R}   }

\newcommand{\eps}{\varepsilon}

\newcommand{\e}{  \text{e}   }

\newcommand{\Z}{  \mathbb{Z}   }
\newcommand{\N}{  \mathbb{N}   }

\renewcommand{\H}{  \mathcal{H}   }

\newcommand{\T}{  \mathbb{T}   }

\newcommand{\p}{  \partial   }

\newcommand{\dis}{  \displaystyle   }

\renewcommand{\a}{  \alpha   }
\renewcommand{\p}{  \bf p  }

\newcommand{\s}{  \sigma   }

\renewcommand{\>}{  \rangle   }
\renewcommand{\phi}{  \varphi   }

\numberwithin{equation}{section}

\author{ Nicolas Burq}
\address{Laboratoire de Math\'ematiques, UMR 8628 du CNRS, B\^at. 425,
Universit\'e Paris Sud, 91405 Orsay Cedex, France and Ecole Normale Sup\'erieure, 45, rue d'Ulm, 75005 Paris,  Cedex 05,  France, UMR 8553 du CNRS }
\email{nicolas.burq@math.u-psud.fr}
\author{ Laurent Thomann }
\address{Laboratoire de Math\'ematiques J. Leray, UMR  6629 du CNRS, Universit\'e de Nantes, 
2, rue de la Houssini\`ere,
44322 Nantes Cedex 03, France}
\email{laurent.thomann@univ-nantes.fr}
\author{ Nikolay Tzvetkov}
\address{University of Cergy-Pontoise, UMR CNRS 8088, Cergy-Pontoise, F-95000}
\email{nikolay.tzvetkov@u-cergy.fr}

\title[Global infinite energy solutions  for the   cubic wave equation ]{Global infinite energy solutions  for the   cubic wave equation} 

\begin{document}
\frontmatter
 \begin{abstract}
 We prove the existence of infinite energy global solutions of the cubic wave equation in dimension greater than $3$.
 The data is a typical element on the support of suitable probability measures.
  \end{abstract}
\thanks{L.T. was partly  supported by the  grant ANR-10-JCJC 0109   and N.T. by an  ERC grant.}
\maketitle
\mainmatter

 \section{Introduction}
This paper is a higher dimensional sequel of the recent article \cite{BTproba} by the first and the third authors (and also of \cite{BT2, BT3,BL}).
As such it aims to construct global in time solutions of the cubic wave equation with low regularity (infinite energy) random initial data. To the best of our knowledge such a regularity is out of reach of the present deterministic methods. The major difference between the present paper and \cite{BTproba} is that here we only establish existence results and in particular no uniqueness statement is proven. Let us recall that in \cite{BTproba} a suitable uniqueness and a probabilistic continuity of the flow were proven. This result was followed by more recent results by Nahmod-Pavlovic-Staffilani~\cite{NPS} on the $2$ and $3$-dimensional homogeneous Navier-Stokes equation, where the authors obtain strong (in $2$-d) and weak (in $3$-d) results, and in turn, here we are inspired by this latter $3$-d weak-existence result. 
Related weak-existence results  had been already used in the context of the randomly forced Navier-Stokes equation by Da Prato-Debussche~ \cite{DPD} and the Euler equation by Albeverio-Cruzeiro~\cite{AC}, using more sophisticated probabilistic tools (Prokhorov and Skorohod Theorems).
This approach may be seen as the analogue in the random setting of the Leray compactness method for constructing solutions of nonlinear evolution equations. 
It has the advantage to require less regularity on the initial data,  one allows infinite energy while the Leray method requires finite energy of the data.
It should however be emphasised that as in the Leray method our approach still makes a crucial use of the energy functional.  
In this paper we will only need an invariance property for the linear evolution combined with large deviation estimates on the nonlinear part which are much easier to achieve than the invariance properties as in \cite{DPD,AC}.
Let us now describe our model.
Let $d\geq 3$ and consider the cubic wave equation on the torus $\T^{d}=(\R/2\pi\Z)^{d}$
 \begin{equation}\label{Wv0}
\left\{
\begin{aligned}
&\partial^{2}_{t}u- {\mathbf{\Delta}} u+ u^{3}=0, \quad   (t,x)\in \R\times \T^{d},\\
&(u,\partial_{t}u)(0,\cdot)=  (u_{0},u_{1}) \in \H^{s},
\end{aligned}
\right.
\end{equation}
where ${\mathbf{\Delta}}:={\mathbf{\Delta}}_{\T^{d}}$ is the Laplace operator and 
$$\H^{s}=\H^{s}(\T^{d}):=H^{s}(\T^{d})\times  H^{s-1}(\T^{d}).$$
Denote by $s_{c}=(d-2)/2$ the critical (scaling) Sobolev index for \eqref{Wv0}. Then one can show that \eqref{Wv0} is well-posed in $\H^{s}$ for $s>s_{c}$ (\cite{GiVe}) and ill-posed when $s<s_{c}$ (\cite{GiVe,CCT,L}). See the introduction of \cite{BTproba} for more details.
The energy of \eqref{Wv0} reads 
\begin{equation*}
\mathcal{E}(u)=\frac12\int_{\T^{d}}\big(|\nabla u|^{2}+(\partial_{t} u)^{2}\big)+\frac14\int_{\T^{d}}u^{4},
\end{equation*}
thus with deterministic compactness methods due to Leray (see {\it e.g.} Lebeau \cite[Section 6]{L} for the application of the method in the context of \eqref{Wv0}), 
we can construct   global weak solutions to \eqref{Wv0} so that 
\begin{equation*}
\big(u,\partial_{t}u\big)\in \mathcal{C}_{w}\big(\R; H^{1}(\T^{d})\cap L^{4}(\T^{d})\big)\times \mathcal{C}_{w}\big(\R ;  L^{2}(\T^{d})\big),
\end{equation*}
(here $\mathcal{C}_{w}$ means weak continuity in time) and $\mathcal{E}(u)(t)\leq \mathcal{E}(u)(0)$ for all $t\in \R$.
Observe that for $d>4$ one has $1<s_{c}$  and thus the construction of weak solutions works for data of supercritical regularity with respect to the scaling of the equation.
However it requires finite energy of the initial data. The main goal of this paper is to show that weak solutions still exist for infinite energy, almost surely with respect to a large class of probability measures.

Let us now describe precisely the initial data sets (statistical ensembles) that we shall consider in this article. Here we follow \cite{BTproba}. Let $0<s<1$ and let $(u_{0},u_{1})\in \H^{s}$ with Fourier series
\begin{equation*} 
u_{j}(x)=a_{j}+\sum_{n\in \Z_{\star}^{d}}\big(b_{n,j}\cos(n\cdot x)+c_{n,j}\sin(n\cdot x)\big),\quad j=0,1,
\end{equation*}
where $\Z_{\star}^{d}=\Z^{d}\backslash\{0\}$. Then let $\big(\alpha_{j}(\omega),\beta_{n,j}(\omega),\gamma_{n,j}(\omega)\big)$, $n\in \Z_{\star}^{d}$, $j=0,1$ be a sequence of independent real random variables given on a probability space $(\Omega,\mathcal{F},\p)$ with a joint distribution $\theta$ satisfying 
\begin{equation*}
\exists\,c>0,\quad \forall\,\gamma\in \R,\quad \int_{-\infty}^{\infty}\e^{\gamma x}\text{d}\theta(x)\leq \e^{c \gamma^{2}}.
\end{equation*}
We then  define the random variables $u_{j}^{\omega}$ by
\begin{equation*}
u_{j}^{\omega}(x)=\alpha_{j}(\omega)a_{j}+\sum_{n\in \Z_{\star}^{d}}\big(\beta_{n,j}(\omega)b_{n,j}\cos(n\cdot x)+\gamma_{n,j}(\omega)c_{n,j}\sin(n\cdot x)\big),
\end{equation*}
and we define the measure $\mu_{(u_{0},u_{1})}$ on $\H^{s}$ as the image of $\p$ under the map 
\begin{equation*}
\dis  \omega\longmapsto (u^{\omega}_{0},u^{\omega}_{1})\in \H^{s}.
 \end{equation*}
We then define  $\mathcal{M}^{s}$ by  
\begin{equation*}
\mathcal{M}^{s}=\bigcup_{(u_{0},u_{1})\in \H^{s}}\big\{\mu_{(u_{0},u_{1})}\big\}.
\end{equation*}
Denote by
  \begin{equation}\label{linear}
\dis S(t)(u_{0},u_{1})=\cos\big(\,t\sqrt{-{\mathbf{\Delta}}}\,\big)(u_{0})+\frac{\sin\big(\,t\sqrt{-{\mathbf{\Delta}}}\,\big)}{\sqrt{-{\mathbf{\Delta}}}}(u_{1}),
\end{equation}
the free wave  evolution. Then our result reads 
 \begin{theo}\label{theo1} 
Let $0<s<1$ and $\mu\in \mathcal{M}^{s}$.  Then there exists a set $\Sigma$ 
of full $\mu$ measure   so that for every $(u_{0},u_{1})\in \Sigma\subset \H^{s}$ the
equation \eqref{Wv0} with 
initial condition $(u(0),\partial_{t}u(0))=(u_{0},u_{1})$ has a   solution 
\begin{equation*} 
u(t)=S(t)(u_{0},u_{1})+w(t)  ,
\end{equation*}
where for any $\eps>0$
\begin{equation*}
\big(w,\partial_{t}w\big)\in \mathcal{C}\big(\R; H^{1-\eps}(\T^{d})\times H^{-\eps}(\T^{d}) \big).
\end{equation*}
Moreover, for all $t\in \R$
 \begin{equation*}\begin{gathered}
 \|(w(t), \partial_t w(t)) \|_{\mathcal{H}^1(\T^{d})} \leq C (M+ |t|)^{\frac {1-s} s + \varepsilon} ,\\
\|w(t)\|_{L^4( \T^d)}\leq C (M+ |t|)^{\frac {1-s} {2s} + \varepsilon},
\end{gathered}
\end{equation*} 
with 
$ \mu (M>\lambda) \leq C e^{-\lambda^\delta}$ for some $\delta>0$.

\end{theo}
\begin{rema}
Let us recall (see \cite{BTproba}) that if the  measure $\mu \in \mathcal{M}^s$ is constructed using data $(u_0,u_1) \in \mathcal{H}^s (\mathbb{T}^d)$, then $\mu( \mathcal{H}^s)=1$, while if for some $s<\sigma$, we have $(u_0,u_1) \notin \mathcal{H}^\sigma (\mathbb{T}^d)$, then as soon as the random variables $(\alpha_j, \beta_{n,j}, \gamma_{n,j})$ do not accumulate at $0$ (for example, in the case where they are non trivial and identically distributed, then  $\mu( \mathcal{H}^{\s})=0$. On the other hand, under rather weak assumptions, $\mu(B^{s})>0$ for any non empty open  ball $B^{s}\subset \H^{s}$ (see  \cite[Proposition 1.2]{BTproba}).
\end{rema}
Let us now mention two possible extensions of our result.
In the case $d=4$ one may expect to get uniqueness by combining the analysis of \cite{BTproba} with the critical $H^1$ theory for \eqref{Wv0}. 
One may also expect to include the case $s=0$ by elaborating on the arguments developed in \cite{BTproba} to treat this case. It is not clear to us what 
happens for $s<0$ (and in \cite{BTproba} as well). In particular we do not know whether  $s=0$ is the optimal regularity one may achieve by our approach.
Invariant Gibbs measures for dispersive equations were extensively studied (see {\it e.g.} \cite{Zhidkov,Bourgain1,Bourgain2,Tzvetkov1,Tzvetkov2,Oh1,Oh2,BTT} ).
In these papers the Gibbs measure is combined with a suitable local in time result (which can sometimes be quite involved) 
to get global existence and uniqueness on the support of the measure.
By an extension of the method (using in particular Skorohod and Prokhorov theorems) we use in this paper one may construct a dynamics (without any uniqueness) on the support of a Gibbs measure and prove its invariance. We plan to give several relevant examples of this observation in \cite{BTT3}. 
We however do not see how to make work such an approach in the context of \eqref{Wv0}. Indeed, the present methods of renormalization of Gibbs measures are restricted to dimensions $\leq 2$ (see~\cite{Bourgain1}).  
Let us also recall that as mentioned above a global existence based on Gibbs measures only works for a very specific choice of the initial distribution. 
On the other hand, it has of course the advantage to give a quite remarkable dynamical property of the flow.

The rest of the paper is organised as follows.
In Section~\ref{Sect.3} we recall stochastic properties of the linear flow which were proven in \cite{BTproba}. 
In Section \ref{Sect.4} we study the dynamics of an approximation of \eqref{Wv0}.  Section \ref{Sect.5} is devoted to the proof of Theorem \ref{theo1}. 
\begin{acknowledgements}
We thank Arnaud Debussche for discussions and for pointing out the reference \cite{DPD}. 
The second author is very grateful to Philippe Carmona for many clarifications on measures.
\end{acknowledgements}
  \section{Stochastic estimates on the linear flow}\label{Sect.3}
 Once for all we fix $0<s<1$ and $\mu=\mu_{(u_{0},u_{1})} \in \mathcal{M}^{s}$. 
 Recall the definition \eqref{linear} of the linear wave propagator $S(t)$. 
 In this section we prove estimates which reflect the invariance of $\mu$ under $S(t)$.
 This is the only measure invariance aspect used in this paper.
  \subsection{The projectors}
 Denote by $\Z^{d}_{\star}=\Z^{d}\backslash\{0\}$. For a Fourier series $u$ 
 \begin{equation*}
 u(x)=a+\sum_{n\in \Z_{\star}^{d}}\big(b_{n}\cos(n\cdot x)+c_{n}\sin(n\cdot x)\big),
 \end{equation*}
 we denote by $\Pi_{0}(u)=a$ and for $N\geq 1$
 \begin{equation*}
 \Pi_{N}(u)=a+\sum_{1\leq |n|\leq N}\big(b_{n}\cos(n\cdot x)+c_{n}\sin(n\cdot x)\big) \quad \text{and}\quad \Pi^{N}=1-\Pi_{N}.
 \end{equation*}
 Let $\chi\in \mathcal{C}_{0}^{\infty}(-1,1)$, so that $\chi\equiv 1$ on $(-1/2,1/2)$. Let us also introduce   the smooth spectral projector
\begin{equation*} 
S_{N}(u)
\equiv \chi(-N^{-2}{\mathbf{\Delta}})
=a+\sum_{n\in \Z^{d}_{\star}}\chi\Big(\frac{|n|^2}{N^2}\Big)\big(b_{n}\cos(n\cdot x)+c_{n}\sin(n\cdot x)\big),
\end{equation*}
which will be needed in the next section. This operator has the following property (see {\it e.g.} \cite{BGT2} for a proof). 
\begin{lemm}\label{lem.sn}  Let $M$ be a compact Riemannian manifold. Let ${\mathbf{\Delta}}$ be the Laplace-Beltrami operator on $M$.  
Let $1\leq p\leq \infty$ and denote by $L^{p}=L^{p}(M)$. Then 
$S_{N}= \chi(-N^{-2}{\mathbf{\Delta}}): L^{p}\longrightarrow  L^{p}$ is continuous and there exists $C>0$ so that for every $N\geq 1$, 
\begin{equation*} 
\|S_{N}\|_{L^{p}\to L^{p}}\leq C.
\end{equation*}
Moreover, for all $f\in L^{p}$, $S_{N}f\longrightarrow f$ in $L^{p}$, when $N\longrightarrow +\infty$.
\end{lemm}
 \subsection{The estimates}
 Following \cite{BTproba},   we introduce the following sets for 
 $$\delta > 1/2, \;\widetilde{\delta} > 1/3, \;\check{\delta} > 0, \;\eps>0$$
\begin{eqnarray*}
F_{M}&=&\Big\{(u_{0},u_{1})\ :\;\;\|  \Pi_{M}(u_{0},u_{1})\|_{\H^{1}(\T^{d})} \leq M^{1-s+\eps}\Big\},\\
G_{M}&=&\Big\{(u_{0},u_{1}):\;\;\|  \Pi_{M}(u_{0})\|_{L^{4}(\T^{d})}\leq M^{\eps}\Big\},\\
H_{M}&=&\Big\{(u_{0},u_{1}):\;\;\|  \<t\>^{-\delta} S(t)(\Pi^{M}(u_{0},u_{1}))\|_{L^{2}(\R_{t};L^{\infty}(\T^{d}))}\leq M^{\eps-s}\Big\}\\
K_{M}&=&\Big\{(u_{0},u_{1}):\;\;\|  \<t\>^{-\widetilde{\delta}} S(t)(\Pi^{M}(u_{0},u_{1}))\|_{L^{3}(\R_{t};L^{6}(\T^{d}))}\leq M^{\eps-s}\Big\}\\
R_M &=&\Big\{(u_{0},u_{1})\ :\;\;\| \<t\>^{-\check{\delta}}  S(t)\Pi^{M}(u_{0},u_{1})\|_{L^\infty(\R; L^4( \T^{d}))} \leq M^{\eps-s}\Big\},
\end{eqnarray*}
and $E_M = F_M \cap G_M\cap H_M\cap K_M\cap R_M$.
Then the following result holds true.  
\begin{lemm} 
For any $\eps>0$, there exists $\eps_0>0$ such that there exist $C,c>0$ such that for every $M\geq 1$
\begin{gather*}
\mu(F^{c}_{M})\leq C\e^{-cM^{2\eps_0}},\quad  \mu(G^{c}_{M})\leq C\e^{-cM^{2\eps_0}},\\
 \mu(H^{c}_{M})\leq C\e^{-cM^{2\eps_0}},\quad  \mu(K^{c}_{M})\leq C\e^{-cM^{2\eps_0}}, \quad \mu(R^{c}_{M})\leq C\e^{-cM^{2\eps_0}}.
\end{gather*}
\end{lemm}
\begin{proof} This result is very close to \cite[Lemma 4.2]{BTproba}. Indeed, the only new point is the bound on the measure of $R_M$, whose proof follows the same lines as the proof of the bound on $K_M$, once we notice that by ($1$-d) Sobolev injection, with $p$ sufficiently large and such that $  \check{\delta}> \frac 1 p$, $\sigma > \frac 1 p$, $\sigma <s$,
\begin{multline}
\qquad \| \<t\>^{-\check{\delta}}  S(t)\Pi^{M}(u_{0},u_{1})\|_{L^\infty(\R; L^4( \T^{d}))} \\
\leq C \| (1+ |D_t|)^ \sigma \<t\>^{-\check{\delta}}  S(t)\Pi^{M}(u_{0},u_{1})\|_{L^p (\R; L^4( \T^{d}))}\\
\leq C' \| \<t\>^{-\check{\delta}}  (1+ |D_t|) ^\sigma S(t)\Pi^{M}(u_{0},u_{1})\|_{L^p (\R; L^4( \T^{d}))}\\
\leq C' \| \<t\>^{-\check{\delta}}  (1+ |D_x|) ^\sigma S(t)\Pi^{M}(u_{0},u_{1})\|_{L^p (\R; L^4( \T^{d}))}.
\end{multline}
 \end{proof}
%

\section{Uniform bounds on  the Sobolev norms, $s>0$}\label{Sect.4}
 For $N\gg1$ we consider the following truncation of \eqref{Wv0}
 \begin{equation}\label{Wv*}
\left\{
\begin{aligned}
&\partial^{2}_{t}u_N- {\mathbf{\Delta}} u_N+ S_{N}\big((S_{N}u_N)^{3}\big)=0, \quad   (t,x)\in \R\times \T^{d},\\
&(u_N,\partial_{t}u_N)(0,\cdot)=  (u_{0},u_{1}) \in \H^{s}.
\end{aligned}
\right.
\end{equation}
In fact, equation \eqref{Wv*} is an ODE in low frequencies, and is the linear wave equation in high frequencies. Indeed, if $K$ is large enough so that $\Pi_K S_N = S_N$, then the equation~\eqref{Wv*} is equivalent to the uncoupled system
 \begin{equation*}
\left\{
\begin{aligned}
&\partial^{2}_{t}\Pi_K u_N- {\mathbf{\Delta}} \Pi_K u_N + S_{N}\big((S_{N}u_N)^{3}\big)=0, \quad   (t,x)\in \R\times \T^{d},\\
&(\Pi_K u_N ,\partial_{t}\Pi_Ku_N)(0,\cdot)=  (\Pi_Ku_{0},\Pi_Ku_{1}) ,\\
& (\text{Id} - \Pi_K) (u_N) = S(t)\big(\,(\text{Id}-\Pi_K)u_{0},(\text{Id}-\Pi_K)u_{1}\,\big). 
\end{aligned}
\right.
\end{equation*}
Then from the conservation of the energy
\begin{equation*}
\mathcal{E}_{N}(\Pi_K(u_N))(t)=\frac12\int_{\T^{d}}\Big((\partial_{t}\Pi_Ku_N)^{2}+|\nabla_{x}\Pi_Ku_N|^{2}+\frac12(S_{N}u_N)^{4}\Big)\text{d}x,
\end{equation*}
we deduce that, for all $N\geq 1$, \eqref{Wv*} admits a global flow $\Phi_{N}(t)$.
The goal of this section is to prove the following statement.
\begin{prop}\label{th.3}
Let $0<s<1$ and $\mu\in {\mathcal M}^s$. Then for any $\varepsilon>0$ there exist $C, \delta >0$ such that for every $(v_0, v_1)\in \Sigma$, 
there exists $M>0$ such that the family of global solution $(u_N)_{N\in \mathbb{N}}$ to~\eqref{Wv*} satisfies 
\begin{equation*}\begin{gathered}
 u_N(t)= S(t) \Pi^0(v_0, v_1)+ w_N(t), \\
 \|(w_N(t), \partial_t w_N(t)) \|_{\mathcal{H}^1} \leq C (M^s+ |t|)^{\frac {1-s} s + \varepsilon} ,\\
\|S_N(u_N)\|_{L^4( \T^d)}\leq C (M^s+ |t|)^{\frac {1-s} {2s} + \varepsilon},
\end{gathered}
\end{equation*} 
with 
$ \mu (M>\lambda) \leq C e^{-\lambda^\delta}\,.
$
\end{prop}
\begin{proof}
We only give the proof for positive times, the analysis for negative times being analogous.
Fix $\varepsilon>0$ and  $\varepsilon_{1}>0$  such that
\begin{equation}\label{eps1eps0}
\varepsilon < \frac s 2, \qquad \frac{1-s+ \varepsilon} {s- 2\varepsilon}\leq
\frac{1-s} {s}+\varepsilon_1,
\end{equation}
and  fix $\delta>1/2, \widetilde \delta > 1/3 $ such that
\begin{equation}\label{obrat}
 (\delta-\frac 1 2)s<2\delta\varepsilon, \qquad \widetilde \delta <1.
 \end{equation}
We have the following statement.
 \begin{lemm}\label{iinntt}
For every $c>0$ there exists $C>0$ such that for every
$t\geq 1$, every integer $M\geq 1$ such that $t\leq cM^{s-2\varepsilon}$, every
$(v_0, v_1)\in E_M$ the solution of \eqref{Wv*} with data $(v_0,v_1)$ satisfies
$$
\|u_N(t)-S(t)\Pi^0(u_{0}, u_{1})\|_{\mathcal{H}^1( \T^d)}\leq CM^{1-s+\varepsilon}.
$$
In particular, thanks to \eqref{eps1eps0}, if $t\approx M^{s-2\varepsilon}$ then
$$
\|u_N(t)-S(t)\Pi^0(u_{0}, u_{1})\|_{\mathcal{H}^1( \T^d)}\lesssim t^{\frac{1-s} {s}+\varepsilon_1}.
$$
\end{lemm}
\begin{proof}
For $(v_0, v_1)\in E_M$   we decompose the solution of \eqref{Wv*} with data $(v_0,v_1)$  as
$$
u_N(t)=S(t) \Pi^M(u_{0}, u_{1})+ w_{N,M},
$$
where $w_{N,M}$ solves the problem
\begin{equation*} 
\left\{
\begin{aligned}
&(\partial_t^2-\Delta_{\T^d})w_{N,M}+S_{N}\big((S_{N}w_{N,M}+S_{N}S(t) \Pi^M(u_{0}, u_{1}))^3\big)=0,\\
& (w_{N,M}(0),\partial_t w_{N,M}(0))=\Pi_{M}(v_0,v_1).
\end{aligned}
\right.
\end{equation*}
 Then thanks to an integration by parts and the fact that $S_{N}$ is self adjoint, we get
\begin{multline}\label{derivee}
\frac{\text{d}}{\text{d}t}\mathcal{E}_{N}(w_{N,M})=\\
\begin{aligned}
&=\int_{\T^{d}}\Big(\partial^{2}_{t}w_{N,M}\partial_{t}w_{N,M}+\nabla_{x}w_{N,M}\cdot \partial_{t}\nabla_{x}w_{N,M}+(S_{N}w_{N,M})^{3}\partial_{t} S_{N}w_{N,M}\Big)\text{d}x\\
&=\int_{\T^{d}}\partial_{t}w_{N,M}\Big(\partial^{2}_{t}w_{N,M}-{\mathbf{\Delta}} w_{N,M}+S_{N}\big((S_{N}w_{N,M})^{3}\big)\Big)\text{d}x\\
&=\int_{\T^{d}}\partial_{t}w_{N,M}\Big(S_{N}\big((S_{N}w_{N,M})^{3}\big)-S_{N}\big((S_{N}S(t)\Pi^{M}(u_{0},u_{1})+S_{N}w_{N,M})^{3}\big)\Big)\text{d}x.
\end{aligned}
\end{multline}
Denote by 
\begin{equation*}
g_{M}(t)=\|    S(t)\Pi^{M}(u_{0},u_{1})\|^{3}_{L^{6}(\T^{d})}\quad \text{and}\quad f_{M}(t)=\|    S(t)\Pi^{M}(u_{0},u_{1})\|_{L^{\infty}(\T^{d})} .\end{equation*}
Therefore from \eqref{derivee} and  the Cauchy-Schwarz inequality, we deduce that  
\begin{multline}
\frac{\text{d}}{\text{d}t}\mathcal{E}_{N}(w_{N,M}) \\
\begin{aligned}
&\leq C \mathcal{E}^{1/2} _{N}(w_{N,M})\|    (S_{N}w_{N,M})^{3}-   \big(S_{N}S(t)\Pi^{M}(u_{0},u_{1})+S_{N}w_{N,M}\big)^{3} \|_{L^{2}(\T^{d})}  \nonumber\\
&\leq C \mathcal{E}^{1/2} _{N}(w_{N,M}) \hfill \big(\|    S(t)\Pi^{M}(u_{0},u_{1})\|^{3}_{L^{6}(\T^{d})}+\|    S(t)\Pi^{M}(u_{0},u_{1})\|_{L^{\infty}(\T^{d})} \|S_{N}w_{N,M}\|_{L^{4}(\T^{d})}^{2}\big)\nonumber
\end{aligned}\\
\leq C \mathcal{E}^{1/2} _{N}(w_{N,M})\Big(g_{M}(t)+f_{M}(t)  \mathcal{E}^{1/2} _{N}(w_{N,M})    \Big),
\end{multline}
and with the Gronwall lemma, we obtain
\begin{eqnarray}
\mathcal{E}^{1/2}_{N}(w_{N,M})(t)&\leq &C \e^{C\int_{0}^{t}f_{M}(\tau)\text{d}\tau}\Big(\mathcal{E}^{1/2}_{N}(w_{N,M})(0)+\int_{0}^{t}g_{M}(\tau)\text{d}\tau\Big)\nonumber\\
&\leq &C \e^{C\int_{0}^{T}f_{M}(\tau)\text{d}\tau}\Big(\mathcal{E}^{1/2}_{N}(w_{N,M})(0)+\int_{0}^{T}g_{M}(\tau)\text{d}\tau\Big):=\mathcal{G}_{M}(T)\label{gronwall}
\end{eqnarray}
(notice that since $w_{N,M}(0)$ does not depend on $N$, the right-hand side in the last inequality is also independent on $N$). 
We now observe that for $(v_0,v_1)\in E_M$
\begin{equation*}
 \Big|\int_0^t g_M(\tau)d\tau \Big|\leq C M^{3(-s+ \varepsilon)}\langle t\rangle ^{3 \widetilde{\delta}}
 \leq 
 CM^{3(-s+ \varepsilon)+3\widetilde{\delta}(s-2\varepsilon)}\leq C,
 \end{equation*}
 provided
 $$
 -s+\varepsilon+\widetilde{\delta}(s-2\varepsilon)\leq 0.
 $$
 The last condition can be readily satisfied according to \eqref{obrat}.
 \par
 Next, we have (using Cauchy-Schwarz inequality in time) that for $(v_0,v_1) \in E_M$,
 \begin{equation*}
 \Big |\int_0^t f_M(\tau)d\tau \Big|\leq \|\langle \tau \rangle ^{- \delta} f_M\|_{L^2( \mathbb{R})} \langle t \rangle ^{\delta+ \frac 1 2}\leq CM^{-s + \varepsilon} \langle t\rangle^{\delta +\frac 1 2}
 \leq C M^{-s + \varepsilon+(\delta+ \frac 1 2 )(s-2\varepsilon)}\leq C,
 \end{equation*}
 provided $-s + \varepsilon+(\delta+ \frac 1 2) (s-2\varepsilon)\leq 0$, a condition which is satisfied thanks to \eqref{obrat}.
 \par
 
For $(v_0, v_1) \in E_M$, we have  
 $$ {\mathcal E}^{1/2} (w_{N,M}(0))\leq C( \|\Pi_M( u_{0}, u_{1})\|_{{\mathcal H}^1}+\|\Pi_{M}(v_0)\|_{L^4}^2)  \leq CM^{1-s+ \varepsilon},
 $$
 and coming back to \eqref{gronwall}, we get 
 \begin{equation}\label{borne.w}
{\mathcal E}^{1/2}(w_{N,M}(t)) \leq C M^{1-s+ \varepsilon}.
\end{equation}
Recall that
$$ 
u_N(t)= w_{N,M}(t) + S(t)\Pi^M(u_{0}, u_{1})=S(t)\Pi^0(u_{0}, u_{1})+w_{N,M}(t)-S(t)\Pi_M\Pi^0(u_{0}, u_{1}).
$$
We have that for a solution to the linear wave equation  the linear energy 
$$ \|\nabla_x u \|_{L^2( \T^d)}^2 + \|\partial_t u \|_{L^2( \T^d)}^2
$$ is independent of time and that if $(u, \partial_tu) $ is orthogonal to constants ($(u, \partial_t u) =  \Pi^0 (u, \partial_t u)$), then this energy controls the $\mathcal{H}^1( \T^d)$-norm, we deduce for $(v_0,v_1)\in E_{M}\subset F_M$ that
$$
\|S(t)\Pi_{M}\Pi^0(u_{0}, u_{1})\|_{\mathcal{H}^1( \T^d)}\leq CM^{1-s+\varepsilon}
$$
and therefore
$$
\|u_N(t)-S(t)\Pi^0(u_{0}, u_{1})\|_{\mathcal{H}^1( \T^d)}\leq CM^{1-s+\varepsilon}\,.
$$
This completes the proof of Lemma~\ref{iinntt}.
\end{proof}
Next we set
$$
E^{M}=\bigcap_{K\geq M} E_{K},
$$ 
where the intersection is taken over the dyadic values of $K$, {\it i.e.} $K=2^j$ with $j$ an integer.
Thus $\mu(E^M)$ tends to $1$ as $M$ tends to infinity.
Using Lemma~\ref{iinntt}, we obtain that  there exists $C>0$ such that for every $t\geq 1$, every $M$, every $(v_0,v_1)\in E^M$, and every $N\in \N$,
$$
\|u_N(t)-S(t)\Pi^0(u_{0}, u_{1})\|_{\mathcal{H}^1 ( \T^d)}\leq C\big(M^{1-s+\varepsilon}+t^{\frac{1-s} {s}+\varepsilon_1}\big)\,.
$$
Furthermore, by~\eqref{borne.w} and the definition of $R_M$, we get that for $(u_0, u_1) \in E_M$, and $t\leq c M^{s-2 \eps}$ 
\begin{eqnarray*}
 \| S_N (u_N)\|_{L^4( \T^d)}(t)
 &\leq &\| S_N (w_{N,M})\|_{L^4(\T^d)}(t)+ \| S_N (S(t) \Pi^M(u_0, u_1))\|_{L^4( \T^d)}(t)\\
& \leq &\mathcal{E}^{1/4} ( w_{N,M})(t)  + M^{-s+2\epsilon}\leq C M^{\frac{1-s+\epsilon}{2}}.
\end{eqnarray*}
Finally, we set 
 $$
 E= \bigcup_{M= 1}^{\infty} E^M\,.
 $$
 We have thus shown the $\mu$ almost sure bounds on the possible growths of the Sobolev norms of the solutions established in the previous section
 for data in $E$ which is of full $\mu$ measure.  This completes the proof of Proposition~\ref{th.3}.
\end{proof}
\section{Passing to the limit}\label{Sect.5}
\subsection{Some deterministic estimates} We now need an interpolation result. Define the space $W_{T}^{1,\infty}$ by the norm $\|u\|_{W_{T}^{1,\infty}}=\|u\|_{L_{T}^{\infty}}+\|\partial_{t}u\|_{L_{T}^{\infty}}$, and denote by $H^{\s}=H^{\s}(\T^{d})$. 
\begin{lemm}\label{lemm.42}
Let $T>0$,  $-\infty<\s_{2}\leq \s_{1} <+\infty$ and assume that 
$$u\in  L^{\infty}\big([-T,T]; H^{\s_{1}}\big), \qquad \partial_{t}u\in  L^{\infty}\big([-T,T]; H^{\s_{2}}\big).$$ Then for  all $\theta\in (0, 1)$, and all $t_{1},t_{2}\in [-T,T]$
\begin{equation*} 
\|u(t_{1})-u(t_{2})\|_{H^{\theta \s_1 + (1- \theta)\s_2}}\leq C|t_{1}-t_{2}|^{1-\theta}\|u\|^{\theta}_{L^{\infty}_{T}H^{\s_{1}}} \|  u\|^{1-\theta}_{W_{T}^{1,\infty}H^{\s_{2}}}.
\end{equation*}
\end{lemm}
\begin{proof}
By H\"older we get
\begin{equation*} 
\|u(t_{1})-u(t_{2})\|_{H^{\s_{2}}}=\|\int_{t_{1}}^{t_{2}}\partial_{\tau}u(\tau)\text{d}\tau\|_{H^{\s_{2}}}\leq |t_{1}-t_{2}|\|\partial_{t} u\|_{L^{\infty}_{T}H^{\s_{2}}}.
\end{equation*}
Next  we clearly have 
\begin{equation*}
\|u(t_{1})-u(t_{2})\|_{H^{\s_{1}}}\leq 2\| u\|_{L^{\infty}_{T}H^{\s_{1}}},
\end{equation*} and we conclude using that 
$$ \|u\|_{H^{\theta\s_1 + (1- \theta)\s_2}}\leq \| u\|_{H^{\s_1}}^\theta \| u \|_{H^{\s_2}}^{1- \theta}.
$$
\end{proof}
Now for $\s\in \R$ and $\a\in (0,1)$, let us define the space $\mathcal{C}_{T}^{\a}H^{\s}=\mathcal{C}^{\a}\big([-T,T]; H^{\s}(\T^{d})\big)$ by the norm 
\begin{equation*}
\|u\|_{\mathcal{C}_{T}^{\a}H^{\s}}=\sup_{t_{1},t_{2}\in [-T,T], t_{1}\neq t_{2}}\frac{\|u(t_{1})-u(t_{2})\|_{H_{x}^{\s}}}{|t_{1}-t_{2}|^{\a}}+\|u\|_{L^{\infty}_{T}H_{x}^{\s}}.
\end{equation*}
 According to  Ascoli theorem, we obtain
 \begin{lemm}\label{lem.4.2}
 For any $T>0$, any $\a>0$ and any $\epsilon >0$, the embedding
 $$\mathcal{C}_{T}^{\a}H^{\s}\mapsto C((0,T); H^{\s - \epsilon})
 $$ is compact.
 \end{lemm}
 
 \subsection{The compactness argument}
 According to Proposition~\ref{th.3}, we know that almost surely, there exists $M\geq 1$ such that  the family of solutions  to~\eqref{Wv*}
 $$u_N(t)= S(t) \Pi^0(v_0, v_1)+ w_N(t), $$  is such that 
 \begin{equation*}
 \begin{aligned}
  \|(w_N(t), \partial_t w_N(t)) \|_{\mathcal{H}^1(\T^d)} \leq C (M^s+ |t|)^{\frac {1-s} {s} + \varepsilon}\\
\|S_N(u_N)\|_{L^4((0,t)\times \T^d)}\leq C (M^s+ |t|)^{\frac {1-s} {2s} + \varepsilon} |t|^{1/4}.
\end{aligned}
\end{equation*}
We apply Lemma \ref{lemm.42} with $\s_{1}=1$ and $\s_{2}=0$ and we deduce that   the sequence $w_N$ is for any $\epsilon>0$ bounded in $\mathcal{C}_{T}^{\epsilon/2}H^{1-\epsilon/2}$. According to Lemma~\ref{lem.4.2} we can almost surely extract a sequence converging for any $T$ in $\mathcal{C}\big((0,T); H^{1 - \epsilon}\big)$, to a limit that we denote by $w$.
On the other hand, the sequence $S_N(u_N)$ is, for any $T$ bounded in $L^4_{t,x}$ and we can consequently extract a sequence converging weakly in $L^4_{loc, t, x}$ to a limit that we denote by $u$. But for any $K\in \N$,  if $K\leq N-2$, we have 
$$S_K (S_N(u_N)) = S_K(u_N)= S_K( S(t)\Pi^0(v_0, v_1)+ w_N(t)),$$
and we deduce that (in distribution sense), $S_K (S_N(u_N))$ is converging to $S_K(u)$ on the  one hand  and to $S_K \bigl(S(t) (u_0, u_1) + w \bigr)$ on the other hand. Hence 
$$\forall\, K \in \N, \quad S_K (u)=S_K\bigl(S(t) (u_0, u_1) + w \bigr).$$ 
We deduce that (in distribution sense) $u=S(t) (u_0, u_1) + w$. 
Now we deduce that $S_N(u_N)$ is converging weakly in $L^4_{loc,t,x}$ and strongly in $L^2_{loc,t,x}$ to $u$ (here by strong convergence in $L^p_{loc,t,x}$ we mean that the convergence is strong on any compact set). By interpolation, we deduce that $S_N(u_N)$ is converging strongly  to $u$ in $L^p_{loc,t,x}$ for $2\leq p<4$. In particular using this property for $p=3$, we can pass to the limit in~\eqref{Wv*} (here we use Lemma \ref{lem.sn} to pass to the limit in the nonlinear term) and obtain that $u$ satisfies~\eqref{Wv0}.  To prove the convergence of $\partial_{t}w_{N}$ in $ \mathcal{C}\big((0,T);  H^{-\eps}(\T^{d}) \big)$, we estimate 
$$\partial_t^2 w_{N}=\Delta w_{N}-S_{N}\big((S_{N}w_{N}+S_{N}S(t) \Pi^0(u_{0}, u_{1}))^3\big),$$
in $ L^{\infty}\big((0,T);  H^{-\tau}(\T^{d}) \big)$ with $\tau=\max{(d/4,1)}$ (here we use $L^{4/3}(\T^{d})\subset H^{-d/4}(\T^{d})$), and we can conclude thanks to  Lemma \ref{lemm.42} with $\s_{1}=0$ and $\s_{2}=-\tau$.


\begin{thebibliography}{9}

\bibitem{AC} 
S. Albeverio, A. Cruzeiro.
\newblock Global flows with invariant (Gibbs) measures for Euler and Navier-Stokes two dimensional fluids.
\newblock {\em Comm. Math. Phys.} 129 (1990) 431--444.

 
\bibitem{Bourgain2}
J. Bourgain.
\newblock Periodic nonlinear Schr\"odinger equation and invariant measures.
\newblock {\em Comm. Math. Phys.} 166 (1994) 1--26.


\bibitem{Bourgain1}
J.~Bourgain.
\newblock Invariant measures for the 2D-defocussing nonlinear Schr\"odinger equation.
\newblock {\em Comm. Math. Phys.}, 176 (1996) 421--445.

  \bibitem{BL}
 N.~Burq and  G. Lebeau.
 \newblock Injections de Sobolev probabilistes et applications (2011) {\em http://arxiv.org/abs/1111.7310} 

 
 \bibitem{BGT2}
N.~Burq, P.~G{\'e}rard and N.~Tzvetkov.
\newblock Strichartz inequalities and the nonlinear {S}chr\"odinger
              equation on compact manifolds.
\newblock {\em Amer. J. Math.}, 126,  no. 3, 569--605, 2004.
%
 \bibitem{BTT}
 N.~Burq, L.~Thomann and  N.~Tzvetkov.
 \newblock On the long time dynamics for the 1D NLS.
 \newblock{\em   arXiv:1002.4054.} , to appear in Ann. Inst. Fourier.
%
\bibitem{BTT3}
N.~Burq, L.~Thomann and  N.~Tzvetkov.
 \newblock Remarks on the Gibbs measures for nonlinear dispersive equations,
 in preparation.
%
 \bibitem{BTproba}
 N.~Burq and  N.~Tzvetkov.
 \newblock Probabilistic well-posedness for the cubic wave equation (2011).
\newblock{\em arXiv:1103.2222.},  to appear in JEMS.

\bibitem{BT2}
N.~Burq, N.~Tzvetkov.
\newblock Random data Cauchy theory for supercritical wave equations \nolinebreak[4 ] I: local
existence theory.
\newblock{\em Invent. Math.} 173, No. 3, (2008), 449--475.
%
\bibitem{BT3}
 N.~Burq and  N.~Tzvetkov.
 \newblock Random data Cauchy theory for supercritical wave equations \nolinebreak[4 ] II:  A global existence result.
 \newblock{\em Invent. Math.} 173, No. 3,  (2008), 477--496.
 %
  \bibitem{CCT} 
  M.~Christ,  J.~Colliander and  T.~Tao, 
   \newblock  Ill-posedness for nonlinear Schr\"odinger and wave equations.
\newblock{\em  arXiv:0311048}.
  %
\bibitem{DPD}
 G. Da Prato and A. Debussche.
 \newblock Two-dimensional Navier-Stokes equations driven by a space-time white noise.
 \newblock{\em  J. Funct. Anal.}  196  (2002),  no. 1, 180--210. 
  
 \bibitem{GiVe} 
 J.~Ginibre, G. Velo.
\newblock Generalized Strichartz inequalities for the wave equation
\newblock {\em  J. Funct. Anal. } 133 (1995), 1, 50--68.

 
 \bibitem{L} 
 G.~Lebeau.
 \newblock   Perte de r\'egularit\'e pour les \'equation d'ondes sur-critiques.
  \newblock {\em Bull. Soc. Math. Fr. } 133 (2005), 145--157.
  
 
\bibitem{NPS}
A. R. Nahmod, N. Pavlovic and  G. Staffilani. 
\newblock Almost sure existence of global weak solutions for super-critical Navier-Stokes equations (2012).
\newblock {\em arXiv:1204.5444.} 

\bibitem{Oh1}
T. Oh.
 \newblock Invariance of the Gibbs measure for the Schr\"odinger-Benjamin-Ono system.
 \newblock {\em SIAM J. Math. Anal.}, 41 (2009), no. 6, 2207--2225.
 %
\bibitem{Oh2}
T. Oh.
 \newblock Invariant Gibbs measures and a.s. global well-posedness for coupled KdV systems.
 \newblock {\em  Diff. Integ. Eq.}, 22 (2009), no. 7-8, 637--668. 


 

\bibitem{Tzvetkov2}
N.~Tzvetkov.
\newblock  Invariant measures for the defocusing NLS.
\newblock {\em  Ann. Inst. Fourier}, 58 (2008) 2543--2604.

\bibitem{Tzvetkov1}
N.~Tzvetkov.
\newblock  Invariant measures for the Nonlinear Schr\"odinger equation on the disc.
\newblock {\em   Dynamics
of PDE.} 3 (2006) 111--160. 

 \bibitem{Zhidkov}
 P. Zhidkov.
 \newblock KdV and nonlinear Schr\"odinger equations : Qualitative theory, Lecture
Notes in Mathematics 1756, Springer 2001.
\end{thebibliography}
\end{document}